\documentclass[12pt,leqno]{amsart}
\usepackage[utf8]{inputenc}
\usepackage[shortlabels]{enumitem}
\usepackage{microtype}
\usepackage{hyperref}
\usepackage{tikz}
\usepackage{amsmath,amssymb,amsthm}
\usepackage[all]{xy}

\usepackage[left=3.75cm,right=3.75cm,top=4cm,bottom=4cm]{geometry}

\newtheorem{theorem}{Theorem}[section]
\newtheorem{lemma}[theorem]{Lemma}

\newtheorem{proposition}[theorem]{Proposition}
\newtheorem{corollary}[theorem]{Corollary}
\theoremstyle{definition}
\newtheorem{definition}[theorem]{Definition}
\newtheorem{example}[theorem]{Example}

\theoremstyle{remark}
\newtheorem{remark}[theorem]{Remark}

\numberwithin{equation}{section}

\def\<{\langle}
\def\>{\rangle}

\newcommand{\diam}{\mathop{\mathrm{diam}}\nolimits}
\newcommand{\dist}{\mathop{\mathrm{dist}}\nolimits}

\newcommand{\Natural}{\mathbb N}

\newcommand{\Real}{\mathbb R}

\newcommand{\abs}[1]{\left\vert#1\right\vert}
\newcommand{\set}[1]{\left\{#1\right\}}
\newcommand{\sequence}[1]{\left\{#1\right\}}
\newcommand{\restricted}{\mathord{\upharpoonright}}

\newcommand{\ext}{\operatorname{ext}}
\newcommand{\strexp}{\operatorname{strexp}}

\newcommand{\norm}[1]{\left\Vert#1\right\Vert}
\newcommand{\duality}[1]{\left\langle#1\right\rangle}
\newcommand{\clco}{\mathop{\overline{\mathrm{co}}}\nolimits}
\newcommand{\cconv}{\clco}
\newcommand{\closedball}[1]{B_{#1}}

\newcommand{\indicator}[1]{{\mathbf 1}_{{#1}}}

\newcommand{\Free}{{\mathcal F}}
\newcommand{\Lip}{{\mathrm{Lip}}_0}
\newcommand{\justLip}{{\mathrm{Lip}}}
\newcommand{\lip}{{\mathrm{lip}}_0}
\newcommand{\F}{\mathcal F}

\newcommand{\N}{\mathbb N}
\newcommand{\R}{\mathbb R}

\newcommand{\pten}{\ensuremath{\widehat{\otimes}_\pi}}

\newcommand\conv{\operatorname{co}}

\title{Extremal structure and Duality of Lipschitz free spaces}

\author[L. Garc\'ia-Lirola]{Luis Garc\'ia-Lirola}
\thanks{The research of L. Garc\'ia-Lirola was supported by the grants MINECO/FEDER MTM2014-57838-C2-1-P and Fundaci\'on S\'eneca CARM
19368/PI/14}
\address[L. Garc\'ia-Lirola]{Universidad de Murcia, Facultad de Matem\'aticas, Departamento de Matem\'aticas, 
30100 Espinardo (Murcia), Spain}
\email{luiscarlos.garcia@um.es}

\author[C. Petitjean]{Colin Petitjean}\thanks{The research of C. Petitjean and A. Proch\'azka was supported by the French “Investissements d’Avenir” program, project ISITE-BFC".}
\address[C. Petitjean]{Universit\'e Bourgogne Franche-Comt\'e, Laboratoire de Math\'ematiques UMR 6623, 16 route de Gray,
25030 Besan\c con Cedex, France}\email{colin.petitjean@univ-fcomte.fr}

\author[A. Proch\'azka]{Anton\'in Proch\'azka}
\address[A. Proch\'azka]{Universit\'e Bourgogne Franche-Comt\'e, Laboratoire de Math\'ematiques UMR 6623, 16 route de Gray,
25030 Besan\c con Cedex, France}
\email{antonin.prochazka@univ-fcomte.fr}

\author[A. Rueda Zoca]{ Abraham Rueda Zoca }\thanks{The research of A. Rueda Zoca was supported by a research grant Contratos predoctorales FPU del Plan Propio del Vicerrectorado de Investigaci\'on y Transferencia de la Universidad de Granada, by MINECO (Spain) Grant MTM2015-65020-P and by Junta de Andaluc\'ia Grants FQM-0185.}
\address[A. Rueda Zoca]{Universidad de Granada, Facultad de Ciencias.
Departamento de An\'{a}lisis Matem\'{a}tico, 18071-Granada
(Spain)} \email{ abrahamrueda@ugr.es}
\urladdr{\url{https://arzenglish.wordpress.com}}



\date{July, 2017}

\begin{document}

\begin{abstract}
We analyse the relationship between different extremal notions in Lipschitz free spaces (strongly exposed, exposed, preserved extreme and extreme points). We prove in particular that every preserved extreme point of the unit ball is also a denting point. We also show in some particular cases that every extreme point is a molecule, and that a molecule is extreme whenever the two points, say $x$ and $y$, which define it satisfy that the metric segment $[x, y]$ only contains $x$ and $y$. The most notable among them is the case when the free space admits an isometric predual with some additional properties. As an application, we get some new consequences about norm-attainment in spaces of vector valued Lipschitz functions.
\end{abstract}

\maketitle

\section{Introduction}

The Lipschitz free space $\mathcal F(M)$ of a metric space $M$ (also known as Arens-Eells space) is a Banach space such that every Lipschitz function on $M$ admits a canonical linear extension defined on $\Free(M)$ (see below for details).
This fundamental linearisation property makes of Lipschitz free spaces a precious magnifying glass to study Lipschitz maps between metric spaces, and for example it relates some well known open problems in the Banach space theory to some open problems about Lipschitz-free spaces (see \cite{godsurvey}).
A considerable effort to study the linear structure and geometry of these spaces has been undergone by many researchers in the last two or three decades.

In the present paper we want to focus on the extremal structure of $\Free(M)$.
The study of extremal structure plays an important role in optimisation  (indeed we obtain some consequences in norm-attainment of Lipschitz maps). 
It has probably started in~\cite{weaver}. 
It is proved for instance in Weaver's book that preserved extreme points of the unit ball $B_{\F(M)}$ are always molecules (i.e. measures of the form $m_{xy}:=(\delta(x)-\delta(y))/d(x,y)$). 
Recently Aliaga and Guirao pushed further this work (see \cite{AG}). 
In particular, answering a question of Weaver, they showed in the compact case that the extreme points are in fact preserved, and are exactly the molecules $m_{xy}$ for which there are no points except $x$ and $y$ in the metric segment $[x,y]$. 
They also give a metric characterisation of preserved extreme points in full generality, which we prove also here by a different argument. More results in the same line appeared in \cite{gpr}, where a metric characterisation of the strongly exposed points is given.

However, the two main questions in this domain remain open: \\
a) If $\mu \in \ext(B_{\Free(M)})$, is $\mu$ necessarily of the form $\mu=m_{xy}$ for some $x\neq y \in M$? \\
b) If the metric segment $[x,y]$ does not contain any other point of $M$ than $x$ and $y$, is $m_{xy}$ an extreme point of $B_{\Free(M)}$ ?

The goal of the present article is to continue the effort in exploring the extremal structure of $\F(M)$ and provide affirmative answers to both previous questions a) and b) in some particular cases. 
For instance, we prove that for the following chain of implications
$$\mbox{strongly exposed} \stackrel{(1)}{\Longrightarrow}\mbox{denting} \stackrel{(2)}{\Longrightarrow} \mbox{preserved extreme}\stackrel{(3)}{\Longrightarrow} \mbox{extreme},$$
the converse of (2) holds true in general (Theorem~\ref{prop:preserveddenting}) but the converse of (1) and (3) are both false (Examples~\ref{ex:TreeCompact} and~\ref{example:AG} respectively). 
However, some of the previous implications are equivalences in some special classes of metric spaces. 
The most notable among them is the case when $\F(M)$ admits an isometric predual with some additional properties. 
We are thus led to the study of preduals of free spaces which seems interesting on its own (see Section \ref{s:Duality}).

The paper is organised as follows. 
In Section \ref{s:general} we prove that every preserved extreme point of $B_{\F(M)}$ is also a denting point in full generality (Theorem \ref{prop:preserveddenting}) and we provide a different proof of the metric charactisation of preserved extreme points given in \cite{AG}. 
We also show that the canonical image $\delta(M)$ of $M$ inside $\Free(M)$  as well as the set $V$ of molecules are weakly closed in $\F(M)$ (Proposition \ref{p:DeltaMweaklyClosed} and Proposition \ref{prop:Vweaklyclosed} respectively). 
Next in Section~\ref{s:Duality}, based on (\cite{dal1,Kalton04,weaver}), we study under what circumstances $\F(M)$ is isometric to a dual space. 
We also pin down a distinguished class of preduals, called natural preduals (see Definition \ref{def:natural}), which turns out to be of particular interest in the later sections. 
In Section \ref{s:ExtStrucNatPred}, we study the extremal structure of spaces admitting such a natural predual. 
In particular, we show under an additional assumption that the set of extreme points coincides with the set of strongly exposed points (Corollary \ref{cor:extnatpredual}). 
Then in Section \ref{section:unifdisc} we focus on the case when $M$ is uniformly discrete and bounded. 
Under this assumption, the question b) has an affirmative answer (Proposition~\ref{unifdisccharext}), 
the implication (1) admits a converse (Proposition~\ref{presestrunidisc}),  
and the question a) has also an affirmative answer if moreover $\F(M)$ admits a natural predual (Proposition~\ref{prop:unifdiscexteq}). 
In Section \ref{s:compact} we show that the converse of (3) holds for certain compact spaces since the norm $\mathcal F(M)$ turns out to be weak* asymptotically uniformly convex. To finish, in Section \ref{Section:normattainment} we apply our work to deduce results about norm attainment of Lipschitz functions. \\

\textbf{Notation.} Throughout the paper we will only consider real Banach spaces. Given a Banach space $X$, we will denote by $B_X$ (respectively $S_X$) the closed unit ball (respectively the unit sphere) of $X$. We will also denote by $X^*$ the topological dual of $X$. The notations $\ext(B_X)$, $\exp(B_X)$, and $\strexp(B_X)$ stand for the set of extreme, exposed, and strongly exposed points of $B_X$, respectively (we refer to \cite{Bourgin83} for formal definitions and background on this concepts). Given a norming subspace $Y$ of $X$, we denote by $\sigma(X, Y)$ the topology on $X$ of pointwise convergence on elements of $Y$. Given a topological space $(T,\tau)$, we denote $\mathcal C_\tau(T)$ the space of continuous functions on $T$.

Given a metric space $M$, $B(x,r)$ denotes the closed ball in $M$ centered at $x\in M$ with radius $r$. We will denote by $\Lip(M,X)$ (or simply $\Lip(M)$ if $X=\R$) the space of all $X$-valued Lipschitz functions on $M$ which vanish at a designated origin $0\in M$. We will consider the norm in $\Lip(M)$ given by the best Lipschitz constant, denoted $\Vert \cdot\Vert_L$. Of particular interest to us is the space of \emph{little-Lipschitz} functions,
\[ \lip(M):=\left\{f\in \Lip(M): \lim\limits_{\varepsilon\rightarrow 0} \sup\limits_{0<d(x,y)<\varepsilon} \frac{| f(x)-f(y)|}{d(x,y)}=0\right\}.\]
We denote $\delta$ the canonical isometric embedding of $M$ into $\mathcal F(M)$, which is given by $\<f,\delta(x)\> = f(x)$ for $x\in M$ and $f\in \Lip(M)$. By a \emph{molecule} we mean an element of $\mathcal{F}(M)$ of the form 
\[ m_{xy}:= \frac{\delta(x)-\delta(y)}{d(x,y)}\]
for $x,y\in M$, $x\neq y$. The set of all molecules in $M$ will be denoted by $V$. Note by passing that $V$ is a norming set for $\Lip(M)$ and so $B_{\mathcal F(M)}=\cconv(V)$. We will need for every $x,y\in M$, $x\neq y$, the function
\[f_{xy}(t):= \frac{d(x,y)}{2}\frac{d(t,y)-d(t,x)}{d(t,y)+d(t,x)}.\]
The properties collected in the next lemma have been proved already in~\cite{ikw2}. 
They make of $f_{xy}$ a useful tool for studying the geometry of $B_{\mathcal F(M)}$. 

\begin{lemma}\label{lemma:IKWfunction} Let $x,y\in M$ with $x\neq y$. We have
\begin{enumerate}[(a)]
\item $\frac{f_{xy}(u)-f_{xy}(v)}{d(u,v)} \leq \frac{d(x,y)}{\max\{d(x,u)+d(u,y),d(x,v)+d(v,y)\}}$ for all $u\neq v \in M$. 
\item $f_{xy}$ is Lipschitz and $\norm{f_{xy}}_{L}\leq 1$. 
\item Let $u\neq v \in M$ and $\varepsilon>0$ be such that $\frac{f_{xy}(u)-f_{xy}(v)}{d(u,v)}>1-\varepsilon$. Then 
\[(1-\varepsilon)\max\{d(x,v)+d(y,v),d(x,u)+d(y,u)\}< d(x,y).\]
\item If $u\neq v \in M$ and $\frac{f_{xy}(u)-f_{xy}(v)}{d(u,v)}=1$, then $u,v\in [x,y]$.
\end{enumerate} 
\end{lemma}

\section{General results}\label{s:general}

Our first goal is to show that every preserved extreme point of $B_{\Free(M)}$ is also a denting point. In order to prove this proposition we need the following characterisation of preserved extreme points which appears in \cite{GMZ14} and that we state for future reference. 

\begin{proposition}[Proposition 9.1 in \cite{GMZ14}]\label{p:CharaPreservedNets} 
Let $X$ be a Banach space and let $x\in B_X$. The following are equivalent:
\begin{enumerate}[(i)]
\item $x$ is an extreme point of $B_{X^{**}}$.
\item The slices of $B_X$ containing $x$ are a neighbourhood basis of $x$ for the weak topology in $B_X$.
\item For every sequences $\{y_n\}$ and $\{z_n\}$ in $B_X$ such that $\frac{y_n+z_n}{2}\stackrel{\Vert\cdot\Vert}{\to}x$ we have that $y_n\stackrel{w}{\to}x$.  
\end{enumerate}
\end{proposition}

It is easy to check that conditions above are also equivalent to the following:
\begin{enumerate}
\item[(iii')] For every $\lambda\in (0,1)$ and sequences $\{y_n\}$ and $\{z_n\}$ in $B_X$ such that $\lambda y_n+(1-\lambda) z_n\stackrel{\Vert\cdot\Vert}{\to}x$ we have that $y_n, z_n\stackrel{w}{\to}x$. 
\end{enumerate}

Next lemma asserts that a net of molecules which converges to a molecule in the weak topology in fact converges in the norm topology. This lemma will be useful in the proof of Theorem \ref{prop:preserveddenting} and also in order to show that the set of molecules is not far from being weakly closed (Proposition \ref{prop:Vweaklyclosed})

\begin{lemma} \label{lemma:wconvmolecules} Assume $\{m_{x_\alpha y_\alpha}\}$ is a net in $V$ which converges weakly to $m_{xy}$. Then $\lim_\alpha d(x_\alpha,x) = 0$ and $\lim_\alpha d(y_\alpha,y) = 0$. 
\end{lemma}

\begin{proof} Assume that $\varepsilon < \min\{d(x,y), \limsup_\alpha d(x_\alpha, x)\}$. Consider the map $f$ given by $f(t) = (\varepsilon-d(x,t))^+$ and let $g=f-f(0)\in \Lip(M)$. Note that $\< g , m_{xy}\> = \frac{\varepsilon}{d(x,y)}>0$. However, 
\[ \liminf_\alpha \< g, m_{x_\alpha y_\alpha}\>  = \liminf_\alpha \frac{-f(y_\alpha)}{d(x_\alpha,y_\alpha)}\leq 0,\]
a contradiction. Therefore, $\lim_\alpha x_\alpha = x$. Analogously we get that $\lim_\alpha y_\alpha= y$. 
\end{proof}

We need the following variation of Asplund\textendash{}Bourgain\textendash{}Namioka superlemma \cite[Theorem 3.4.1]{Bourgin83}. 

\begin{lemma}\label{lemma:superlemmapreserved} Let $A,B \subset X$ be bounded closed convex subsets and let $\varepsilon >0$. Assume that $\diam(A)<\varepsilon$ and  that there is $x_0\in A\setminus B$ which is a preserved extreme point of $\cconv(A\cup B)$. Then there is a slice of $\cconv(A\cup B)$ containing $x_0$ which is of diameter less than $\varepsilon$. 
\end{lemma}

\begin{proof} For each $r\in [0,1]$ let
\[ C_r = \{ x\in X : x = (1-\lambda)y+\lambda z, y\in A, z\in B, \lambda\in [r,1]\}. \]
The proof of the Superlemma says that there is $r$ so that $\diam(\cconv(A\cup B)\setminus \overline{C_r})<\varepsilon$. We will show that $x_0\notin \overline{C_r}$. Thus, any slice separating $x_0$ from $\overline{C_r}$ will do the work. To this end, assume that there exist sequences $\{y_n\}\subset A$, $\{z_n\}\subset B$ and $\lambda_n\subset [r,1]$ such that $x_0 = \lim_n (1-\lambda_n)y_n+\lambda_n z$. By extracting a subsequence, we may assume that $\{\lambda_n\}$ converges to some $\lambda\in [r,1]$. Note that then $x_0=\lim_n (1-\lambda)y_n+\lambda z_n$. Since $x_0$ is a preserved extreme point, this implies that $\{z_n\}$ converges weakly to $x_0$ by Proposition \ref{p:CharaPreservedNets}. That is impossible since $x_0\notin B$ and $B$ is weakly closed as being convex and closed.
\end{proof}

\begin{theorem}\label{prop:preserveddenting} Every preserved extreme point of $B_{\mathcal F(M)}$ is a denting point. 
\end{theorem}

\begin{proof} Let $\mu$ be a preserved extreme point of $B_{\mathcal F(M)}$, which must be an element of $V$. Denote by $\mathcal S$ the set of weak-open slices of $B_{\mathcal F(M)}$ containing $\mu$. Consider the order $S_1\leq S_2$ if $S_2\subset S_1$ for $S_1,S_2\in \mathcal S$. Using (ii) of Proposition~\ref{p:CharaPreservedNets}, every finite intersection of elements of $\mathcal S$ contains an element of $\mathcal S$ and so $(\mathcal S, \leq)$ is a directed set. Assume that $\mu$ is not a denting point. Then, there is $\varepsilon>0$ so that $\diam(S)>2\varepsilon$ for every $S\in \mathcal S$.

We distinguish two cases. Assume first that for every slice $S$ of $B_{\mathcal F(M)}$ there is $\mu_S \in (V\cap S) \setminus B(\mu, \varepsilon/4)$. Then $\{\mu_S\}$ is a net in $V$ which converges weakly to $\mu$.  By Lemma~\ref{lemma:wconvmolecules}, it also converges in norm, which is impossible. 
Thus, there is a slice $S$ of $B_\mathcal F(M)$ such that $\diam(V\cap S)<\varepsilon/2$. Note that
\[ B_{\mathcal F(M)} = \cconv(V) = \cconv( \cconv(V\cap S)\cup \cconv(V\setminus S))\]
and so the hypothesis of Lemma~\ref{lemma:superlemmapreserved} are satisfied for $A=\cconv(V\cap S)$, $B=\cconv(V\setminus S)$, and $\mu\in A\setminus B$. Then there is a slice of $B_{\mathcal F(M)}$ containing $\mu$ of diameter less than $\varepsilon$, a contradiction.  
\end{proof}

Theorem \ref{prop:preserveddenting} provides a new proof of the following result given in \cite{gpr}. 

\begin{corollary} Let $M$ be a length space. Then $B_{\mathcal F(M)}$ does not have any preserved extreme point. 
\end{corollary}
\begin{proof}
The space $\mathcal F(M)$ has the Daugavet property whenever $M$ is a length space \cite{ikw}. In particular, every slice of $B_{\mathcal F(M)}$ has diameter two. Thus, $B_{\mathcal F(M)}$ does not have any denting point. 
\end{proof}

During the preparation of this preprint we have learnt that Aliaga and Guirao~\cite{AG} characterised metrically the preserved extreme points of free spaces. In the following pages we provide an alternative proof of their result which accidentally reproves our Theorem~\ref{prop:preserveddenting}.

\begin{theorem}\label{th:charpreserved} Let $M$ be a metric space and $x,y\in M$. The following are equivalent:
\begin{enumerate}
\item[(i)] The molecule $m_{xy}$ is a denting point of $B_{\mathcal F(M)}$.
\item[(ii)] For every $\varepsilon>0$ there exists $\delta>0$ such that every $z\in M$ satisfies
\[ (1-\delta)(d(x,z)+d(z,y))<d(x,y) \Longrightarrow \min\{d(x,z),d(y,z)\} <\varepsilon.
\]
\end{enumerate}
\end{theorem}

\begin{proof}[Proof of (i)$\Rightarrow$(ii)]
In fact we are going to show that negation of (ii) implies that $m_{xy}$ is not a preserved extreme point.
Since denting points are trivially preserved extreme points, this will show at once that $m_{xy}$ is not denting.

So let us fix $\varepsilon>0$ such that for every $n\in \Natural$ there exists $z_n \in M$ such that 
\[
\left(1-\frac{1}{n}\right) (d(x,z_n)+d(z_n,y))<d(x,y)
\]
but $\min\set{d(x,z_n),d(y,z_n)} \geq \varepsilon$.
Let $\mu$ be a w$^*$-cluster point of $\set{z_n}$ ($\set{z_n}$ is clearly bounded).
By lower semicontinuity of the norm we have
\[
\norm{\delta(x)-\mu}+\norm{\mu-\delta(y)} = d(x,y).
\]
If $\mu \in \set{\delta(x),\delta(y)}$, say  $\mu=\delta(x)$ then by Lemma~\ref{lemma:wconvmolecules} we get that $z_n \to x$ in $(M,d)$ which is a contradiction.

Thus $\mu \notin \set{\delta(x),\delta(y)}$. Then $$\frac{\delta(x)-\delta(y)}{\Vert \delta(x)-\delta(y)\Vert}=\frac{\Vert \delta(x)-\mu\Vert}{\Vert \delta(x)-\delta(y)\Vert}\frac{\delta(x)-\mu}{\Vert \delta(x)-\mu\Vert}+\frac{\Vert \mu-\delta(y)\Vert}{\Vert \delta(x)-\delta(y)\Vert}\frac{\mu-\delta(y)}{\Vert \mu-\delta(y)\Vert}.$$
Thus $\mu$ is a non-trivial convex combination and so it is not preserved extreme which concludes the proof of (i)$\Rightarrow$(ii).
\end{proof}

For the proof of the other implication we need a couple of lemmata. 
The first of them shows that the diameter of the slices of the unit ball can be controlled by the diameter of the slices a subset of the ball that is norming for the dual.
\begin{lemma}\label{lemma:epslice} Let $X$ be a Banach space and let $V\subset S_X$ be such that $B_{X}=\cconv(V)$. 
Let $f\in B_{X^*}$ and $0<\alpha,\varepsilon<1$. Then
\[ \diam(S(f, B_{X}, \varepsilon\alpha)) \leq 2\diam(S(f, B_X, \alpha)\cap V)+4\varepsilon.\]
\end{lemma}

\begin{proof}
Fix a point $x_0 \in S(f, B_X, \alpha)\cap V$. It suffices to show that $\Vert x-x_0\Vert< \diam(S(f, B_X, \alpha)\cap V)+2\varepsilon$ for every $x\in S(f, B_X, \varepsilon\alpha)\cap \conv(V)$. To this end, let $x\in B_X$ be such that $f(x)>1-\varepsilon\alpha$, and $x=\sum_{i=1}^n \lambda_i x_i$, with $x_i\in V$, $\sum_{i=1}^n \lambda_i=1$ and $\lambda_i>0$ for all $1\leq i \leq n$. 
Define
\[ G=\{i\in \{1,\ldots,n\} : f(x_i)>1-\alpha\}\]
and $B=\{1,\ldots,n\}\setminus G$. We have
\begin{align*}
1-\varepsilon\alpha &< f(x) =\sum_{i\in G} \lambda_i f(x_i) + \sum_{i\in B} \lambda_i f(x_i)\\
& \leq \sum_{i\in G} \lambda_i + (1-\alpha)\sum_{i\in B} \lambda_i = 1-\alpha\sum_{i\in B} \lambda_i,
\end{align*}
which yields that $\sum_{i\in B}\lambda_i<\varepsilon$. Now,
\[
\Vert x-x_0\Vert  \leq  \sum_{i\in G} \lambda_i \Vert x_i-x_0\Vert +\sum_{i\in B} \lambda_i \Vert x_i-x_0\Vert \leq \diam(S(f, B_X, \alpha)\cap V) + 2\varepsilon. 
\]
\end{proof}

\begin{lemma}\label{lemma:fdent} Let $x,y\in M$, $x\neq y$ such that $d(x,y)=1$. For every $0<\varepsilon<1/4$ and $0<\tau<1$ there is a function $f\in \Lip(M)$ such that $\Vert f\Vert_L =1$, $\<f, m_{xy}\>>1-4\varepsilon\tau$ and satisfying that for every $u,v\in M$, $u\neq v$, such that if $u,v\in B(x,\varepsilon)$ or $u,v \in B(y,\varepsilon)$, then $\<f,m_{uv}\> \leq 1-\tau$. 
\end{lemma}

\begin{proof} Define $f\colon B(x,\varepsilon) \cup B(y,\varepsilon)\to\mathbb R$ by 
\[f(t)=\begin{cases}
\frac{1}{1+4\varepsilon\tau}(\tau+ (1-\tau)d(y,t)) & \text{ if } t\in B(x,\varepsilon),\\
\frac{1}{1+4\varepsilon\tau}(1-\tau)d(y,t) &\text{ if } t\in B(y,\varepsilon).
\end{cases}
\]
Note that 
\[ \<f, m_{xy}\> = f(x)-f(y) = \frac{1}{1+4\varepsilon\tau} > 1-4\varepsilon\tau.\]
Moreover, note that if $u,v\in B(x,\varepsilon)$ or $u,v\in B(y,\varepsilon)$ then $\<f,m_{uv}\>\leq  \frac{1-\tau}{1+4\varepsilon\tau}\leq 1-\tau$, so the last condition in the statement is satisfied. Now we compute the Lipschitz norm of $f$. It remains to compute $\<f, m_{uv}\>$ with $u\in B(x,\varepsilon)$ and $v\in B(y,\varepsilon)$. In that case we have
\begin{align*} 
|\< f, m_{uv}\>| &= \frac{|\tau+(1-\tau)(d(u,y)-d(v,y))|}{(1+4\varepsilon\tau)d(u,v)}\leq \frac{\tau+(1-\tau)d(u,v)}{(1+4\varepsilon\tau)d(u,v)}\\
&\leq \frac{1}{1+4\varepsilon\tau}\left(\frac{\tau}{1-2\varepsilon}+1-\tau\right) \leq \frac{\tau(1+4\varepsilon)+1-\tau}{1+4\varepsilon\tau}=1\\
\end{align*}
where we are using that $(1-2\varepsilon)^{-1}\leq 1+4\varepsilon$ since $\varepsilon<1/4$. This shows that $\Vert f\Vert_L \leq 1$. Next, find an extension of $f$ with the same norm. Finally, replace $f$ with the function $t\mapsto f(t)-f(0)$. 
\end{proof}

\begin{proof}[Proof of (ii)$\Rightarrow$ (i) of Theorem \ref{th:charpreserved}]
Now, assume that (ii) holds. We can assume that $d(x,y)=1$. Fix $0<\varepsilon <1/4$. We will find a slice of $B_{\mathcal F(M)}$ containing $m_{xy}$ of diameter smaller than $32\varepsilon$. Let $\delta>0$ be given by property (ii), clearly we may assume that $\delta<1$. 
Let $f$ be the function given by Lemma \ref{lemma:fdent} with $\tau=\delta/2$. 
\[ h(t) = \frac{f_{xy}(t) + f(t)}{2}.\]
It is clear that $\Vert h\Vert_L\leq 1$. Moreover, note that
\[ \< h, m_{xy}\>  = \frac{\<f_{xy}, m_{xy}\> + \<f, m_{xy}\>}{2}  > 1-2\varepsilon\tau = 1-\varepsilon\delta.\] 

Take $\alpha=\delta/4$ and consider the slice $S=S(h, B_{\mathcal F(M)}, \alpha)$. Note that $m_{xy}\in S(h,  B_{\mathcal F(M)}, 4\varepsilon\alpha)$. We will show that $\diam(S\cap V)\leq 8\varepsilon$ and as a consequence of Lemma \ref{lemma:epslice} we will get that $\diam S(f, B_{\mathcal F(M)}, \alpha)\leq 32 \varepsilon$. 
So let $u,v\in M$ be such that $m_{uv}\in S$. First, note that $\<f_{xy}, m_{uv}\>>1-\delta$, since otherwise we would have 
\begin{equation*}\label{eq:muv} \< h,m_{uv}\> = \frac{1}{2}(\< f_{xy}, m_{uv}\> +\<f, m_{uv}\>)  \leq \frac{1}{2}(1-\delta)+\frac{1}{2}= 1-\frac{\delta}{2} <1-\alpha
\end{equation*}
Thus, from the property (c) of the function $f_{xy}$ and the hypothesis (ii) we have that
\[
\min\{d(x,u), d(u,y)\} < \varepsilon \quad\mbox{ and }\quad
\min\{d(x,v), d(y,v)\} < \varepsilon.
\]
On the other hand,
\[ 1-\alpha< \<h, m_{uv}\> \leq \frac{1}{2}+\frac{1}{2}\<f, m_{uv}\>\]
and so $\<f, m_{uv}\> > 1-2\alpha=1-\frac{\delta}{2}= 1-\tau$. 
Thus, we have that $u$ and $v$ do not belong simultaneously to neither $B(x,\varepsilon)$ nor $B(y,\varepsilon)$. 
If $d(x,v)<\varepsilon$ and $d(y,u)<\varepsilon$, then it is easy to check that $\<f_{xy}, m_{uv}\>\leq 0$.
So necessarily $d(x,u)<\varepsilon$ and $d(y,v)<\varepsilon$. 
Now, use the estimate 
\begin{align*}
\Vert m_{xy}-m_{uv}\Vert &= \frac{\Vert d(u,v)(\delta(x)-\delta(y))-d(x,y)(\delta(u)-\delta(v))\Vert}{d(x,y)d(u,v)}\\
&\leq \frac{\Vert(\delta(x)-\delta(y))-(\delta(u)-\delta(v))\Vert}{d(x,y)}+\frac{|d(u,v)-d(x,y)|\Vert\delta(u)-\delta(v)\Vert}{d(x,y)d(u,v)}\\
&\leq 2\frac{d(x,u)+d(y,v)}{d(x,y)} \leq 4\varepsilon.
\end{align*}
Therefore, $\diam(S\cap V)\leq 8\varepsilon$. 
\end{proof}

\subsection{Weak topology in free spaces}
The results which follow are independent of the rest of the article. 
The reader interested only in the extremal structure of the free spaces can skip until the end of this section.

A simple examples (Examples \ref{ex:AG} and \ref{ex:NaturalNonNatural} ) show that $\delta(M)$  is not necessarily weak$^*$ closed when $\Free(M)$ is a dual space.
The next proposition shows that the situation is different for the weak topology.

\begin{proposition}\label{p:DeltaMweaklyClosed}
Let $M$ be a complete metric space. 
Then $\delta(M)\subset \Free(M)$ is weakly closed.
\end{proposition}

The proposition could be deduced more or less easily from Proposition~2.1.6 in \cite{weaver} but we propose a self-contained proof.
For the proof we will need the next observation (essentially already present in \cite{weaver}).
The weak* closures of subsets of $\Free(M)$ below are taken in the bidual $\Free(M)^{**}=\Lip(M)^*$.

\begin{lemma}\label{l:WeakStarKey}
Let $M$ be a complete metric space.
Let $\mu \in \overline{\delta(M)}^{w^*}\setminus \delta(M)$.
Then there exists $\varepsilon>0$ such that for all $q_1,\ldots,q_n \in M$ we have that 
\[
\mu \in \overline{\delta\left(M \setminus \bigcup_{i=}^n B(q_i,\varepsilon)\right)}^{w^*}.
\]
\end{lemma}

\begin{proof}
Indeed, otherwise we could find a sequence $(q_n) \subset M$ such that $\mu \in \overline{\delta(B(q_n,2^{-n}))}^{w^*}$ for every $n\geq 1$.
It follows that $\norm{\mu-\delta(q_n)} \leq 2^{-n}$ for every $n$ and thus $\{q_n\}$ is Cauchy. 
By completeness of $M$ it follows that $\mu= \lim_n \delta(q_n) \in \delta(M)$. 
This contradiction proves the claim.
\end{proof}

\begin{proof}[Proof of Proposition~\ref{p:DeltaMweaklyClosed}]
It is enough to show that if $\mu \in \overline{\delta(M)}^{w^*}\setminus \delta(M)$, then $\mu$ is not $w^*$-continuous.
Indeed, this yields that $\mu \notin \Free(M)$ and so $\overline{\delta(M)}^w=\overline{\delta(M)}^{w*}\cap \Free(M)=\delta(M)$.

So let $\mu \in \overline{\delta(M)}^{w*}\setminus \delta(M)$ and let $\varepsilon>0$ be as in Lemma~\ref{l:WeakStarKey}.
Now let $U$ be an open neighborhood of $0$ in $(\closedball{\Lip(M)},w^*)$. 
Since the $w^*$ topology and the topology of pointwise convergence coincide on the ball $\closedball{\Lip(M)}$, we may assume that there are $x_1,\ldots,x_n \in M$ and $\alpha>0$ such that $U=\set{f \in \closedball{\Lip(M)}: \abs{f(x_i)}<\alpha \mbox{ for } i=1,\ldots n}$.
We define $f(x):=\dist(x,\set{x_1,\ldots,x_n})$.
We clearly have $f \in U$.
Moreover since 
\[
\mu \in \overline{\delta\left(M \setminus \bigcup_{i=1}^n B(x_i,\varepsilon)\right)}^{w^*},
\]
we have that $\mu(f) \geq \varepsilon$.
Thus $\mu$ is not weak*-continuous as $V$ was arbitrary.
\end{proof}

We observe the following curious corollary (which also admits an independent proof by combinatorial methods).
\begin{corollary}\label{cor:weaktonormcv}
Let $M$ be a complete metric space. 
If $\{x_n\} \subset M$ is a sequence such that $\delta(x_n)$ converges weakly to some $\mu \in \Free(M)$, then there exists $x \in M$ such that $\mu=\delta(x)$ and $d(x_n,x) \to 0$.
\end{corollary}
\begin{proof}
The fact that $\mu=\delta(x)$ follows from Proposition~\ref{p:DeltaMweaklyClosed}.
For the rest it is enough to pose $f(\cdot):=d(\cdot,x)-d(0,x)$ and use that $d(x_n,x)=\duality{\delta(x_n),f}-\duality{\delta(x),f} \to 0$.
\end{proof}

Given a complete metric space $M$ and $\mu \in \Free(M)\setminus \delta(M)$ there is a weak neighborhood that separates $\mu$ from $\delta(M)$. 
The next example shows that contrary to what one might expect, such a neighborhood is not necessarily of the form $\set{\gamma \in \Free(M): \abs{\duality{f,\gamma-\mu}}<\varepsilon}$ for some $f \in \Lip(M)$ and $\varepsilon>0$.
\begin{example}
Let $M=[0,1]$ with the usual metric and let $\mu$ be the Lebesgue measure on $[0,1]$.
It is well known and can be shown easily using the Riemann sums that $\mu \in \Free(M)$.
It acts on $\Lip([0,1])$ as follows $\duality{\mu,f}=\int_0^1 f(t)dt$.
Now the mean value theorem implies that for every $f \in \Lip(M)$ there exists $x \in [0,1]$ such that 
$\duality{\delta(x),f}=\duality{\mu,f}$.
\end{example}

In light of Proposition \ref{p:DeltaMweaklyClosed}, it is also natural to wonder if the set $V$ of molecules is weakly closed. It is known that $0$ is in the weak-closure of $V$ whenever $M$ is not bi-Lipschitz embeddable in $\mathbb R^N$ (see Lemma 4.2 in \cite{gr}). The following proposition shows that $0$ is the only point that we can reach taking the weak-closure of $V$.

\begin{proposition}\label{prop:Vweaklyclosed}
Let $(M,d)$ be a complete metric space. Then $\overline{V}^{w} \subset V \cup \{ 0 \}$.
\end{proposition}

\begin{proof}
The proof is based on \cite[Theorem 2.5.3]{weaver}. Let us begin with an explanation of this result. To this end, we need to introduce the same kind of map as the one used in the proof of \ref{prop:dualityKalton}. Let $\tilde{M}:=\{(x,y)\in M^2 \, : \, x \neq y \}$ and 
\[
\begin{split}
\Phi: \Lip(M) &\to \mathcal{C}_{b}(\tilde{M})\\
 f(x,y)&\mapsto \frac{f(x) - f(y)}{d(x,y)}
\end{split}
\] (here $\mathcal{C}_{b}(\tilde{M})$ stands for the continuous and bounded functions on $\tilde{M}$). 
It is easy to see that $\Phi$ is an isometry. Now let us denote $\beta \tilde{M}$ the Stone-\v{C}ech compactification of $\tilde{M}$. As usual, we can canonically identify $\mathcal{C}_{b}(\tilde{M})$ with $\mathcal{C}(\beta \tilde{M})$ so that we now see $\Phi$ as a map from $\Lip(M)$ to $\mathcal{C}(\beta \tilde{M})$. Thus $\Phi^*$ goes from $\mathcal{C}(\beta \tilde{M})^* = \mathcal{M}(\beta  \tilde{M})$ to $\Lip(M)^*$. According to Weaver, we say that $\mu \in \Lip(M)^*$ is normal if $\<\mu , f_i \>$ converge to $\<\mu , f \>$ whenever $\{f_i\}$ is a bounded and decreasing  (meaning that $f_i \geq f_j$ for $i\leq j$) net in $\Lip(M)$ which $w^*$-converges to $f \in \Lip(M)$. Clearly normality is implied by $w^*$-continuity. Finally, \cite[Theorem 2.5.3]{weaver} asserts that if $x \in \beta \tilde{M}$ with $\Phi^* \delta(x) \neq 0$, then $\Phi^* \delta(x)$ is normal if and only if $x \in \tilde{M}$.

Let us now prove the assertion of the proposition. 
Since \[\overline{V}^{w}= \overline{V}^{w^*} \cap \F(M) = \{\mu \in \overline{V}^{w^*} \, : \, \mu \mbox{ is } w^* \mbox{-continuous} \},\] 
it is enough to show that if $\mu \in \overline{V}^{w^*} \backslash (V \cup \{0\})$ then $\mu$ is not $w^*$-continuous. 
So let us fix such a $\mu$.  
We identify, as we may,  $\tilde{M}$ with $\delta(\tilde{M}) \subset  \mathcal{M}(\beta  \tilde{M})$. We claim that $\delta(\tilde{M})$ is homeomorphic to $(V,w^*)$. Indeed, it is clear that 
\[
\Phi^*\restricted_{\delta(\tilde{M})}: \delta(x,y) \in \delta(\tilde{M}) \mapsto  m_{xy} \in (V,w^*)
\]  
is continuous and bijective. 
The fact that the inverse mapping is also continuous follows from Lemma \ref{lemma:wconvmolecules}. So the claim is proved. 
Now $(\overline{V}^{w^*},w^*)$ is clearly a compactification of $V$. 
Thus the universal property of the Stone-\v{C}ech compactification provides a surjective extension of $\Phi^*\restricted_{\delta(\tilde{M})}$ that goes from $\delta(\beta\tilde{M})$ to $\overline{V}^{w^*}$. 
It is easy to check that the latter extension is in fact $\Phi^*\restricted_{\delta(\beta\tilde{M})}:  \delta(\beta\tilde{M}) \to \overline{V}^{w^*}$.

Now consider $x \in \beta \tilde{M}$ such that $\Phi^* \delta(x) = \mu \neq 0$. Since $\mu \in \overline{V}^{w^*} \backslash (V \cup \{0\})$, we deduce that $x \in \tilde{M}$. Thus, according to \cite[Theorem 2.5.3]{weaver}, $\Phi^* \delta(x)= \mu$ is not normal and therefore not $w^*$-continuous. This ends the proof.
\end{proof}

From the previous proposition, we deduce a result similar to Corollary \ref{cor:weaktonormcv}.

\begin{corollary}
Let $M$ be a complete metric space. 
If $\sequence{m_{x_ny_n}} \subset \F(M)$ is a sequence of molecules which converges weakly to some $\mu \in \Free(M)$, then there exist $x \neq y \in M$ such that $\mu=m_{xy}$ and $\sequence{m_{x_ny_n}}$ actually converges in norm to $m_{xy}$.
\end{corollary}

\begin{proof}
The fact that $\mu=m_{xy}$ follows from Proposition~\ref{prop:Vweaklyclosed} and the fact that $\sequence{m_{x_ny_n}}$ actually converges in norm follows from Lemma \ref{lemma:wconvmolecules}.
\end{proof}

This last corollary should be compared with \cite[Theorem 5.2]{albiackalton}.

Now that we know that $\overline{V}^w\subset V \cup \set{0}$ we get an easy proof of Weaver's theorem~\cite{weaver} which claims that the preserved extreme points are molecules. 
We include it for completeness as it is directly related to the main subject of this paper.
\begin{corollary}
Let $M$ be a complete metric space and let $\mu$ be a preserved extreme point of $\closedball{\Free(M)}$. 
Then $\mu=m_{xy}$ for some $x\neq y \in M$.
\end{corollary}

\begin{proof}
Indeed, we have that $\clco (V)=B_{\Free(M)}$ and $\overline{B_{\Free(M)}}^{w^*}=B_{\Lip(M)^*}$. 
Thus $\clco^{w^*}(V)=B_{\Lip(M)^*}$ and so by Milman's theorem $\ext(B_{\Lip(M)^*}) \subset \overline{V}^{w^*}$.
Finally we get that $\Free(M)\cap \ext(B_{\Lip(M)^*}) \subset \overline{V}^{w}$ and so Proposition~\ref{prop:Vweaklyclosed} yields $\Free(M)\cap \ext(B_{\Lip(M)^*})\subset V$.
\end{proof}

\section{Duality of some Lipschitz free spaces}\label{s:Duality}
Many of our results in Sections~\ref{s:ExtStrucNatPred} and~\ref{section:unifdisc} use the hypothesis that $\Free(M)$ admits an isometric predual which makes $\delta(M)$ w$^*$-closed.
Even though for some of these results we do not know whether this hypothesis is superfluous, we take the opportunity to study the Lipschitz free spaces which admit such a predual.

\begin{definition} \label{def:natural}
Let $M$ be a bounded metric space. We will say that a Banach space $X$ is a \emph{natural predual} of $\Free(M)$ if $X^*=\Free(M)$ isometrically and $\delta(M)$ is $\sigma(\Free(M),X)$-closed.
\end{definition}

It is obvious that when $M$ is a compact metric space then every isometric predual of $\Free(M)$ is natural.
We will show in Examples~\ref{ex:AG} and~\ref{ex:NaturalNonNatural} that there are isometric preduals to $\Free(M)$ which are not natural.

Let us state for the future reference an almost obvious characterisation of natural preduals.
\begin{proposition}\label{prop:suffnatural} Let $M$ be a bounded metric space and let $X$ be an isometric predual of $\mathcal F(M)$.
Then the following are equivalent:
\begin{enumerate}[(i)]
\item There is a compact Hausdorff topology $\tau$ on $M$ such that $X\subset \Lip(M)\cap \mathcal C_\tau(M)$.
\item $\delta(M)$ is $\sigma(\Free(M),X)$-closed.
\end{enumerate}
\end{proposition}

\begin{proof}
We only need to show (i)$\Rightarrow$(ii). To this end, note that the $w^*$-topology of $\F(M)$ and the $\tau$-topology coincide on $\delta(M)$. Indeed, every $w^*$-open set in $\delta(M)$ is also $\tau$-open since $X$ is made up of $\tau$-continuous functions, so that the $w^*$-topology is weaker than $\tau$ on $\delta(M)$. By compactness of the Hausdorff topology $\tau$, we have that they agree on $\delta(M)$.
\end{proof}

The natural preduals are quite common. In fact, the known constructions of isometric preduals to $\Free(M)$ when $M$ is bounded all produce natural preduals. 
Indeed, this is the case for Theorem~3.3.3 in~\cite{weaver} as well as Theorem~2.1 in~\cite{dal1} because of the compactness.
In the next theorem we will show that it is also true for Theorem~6.2 in~\cite{Kalton04}.
We will say that a subspace $X$ of $\Lip(M)$ \emph{$1$-separates points uniformly (shortened $1$-S.P.U.)} if 
for every $x,y \in M$ and every $\varepsilon>0$ there is $f \in X$ such that $f(x)-f(y)=d(x,y)$ and $\norm{f}_L<1+\varepsilon$. 

\begin{proposition}[Theorem~6.2 in~\cite{Kalton04}]\label{prop:dualityKalton} 
Let $M$ be a separable bounded pointed metric space and let $\tau$ be a topology on $M$ so that $(M, \tau)$ is compact. 
Assume that $X = \lip(M) \cap \mathcal{C}_{\tau}(M)$ 
$1$-S.P.U. 
Then $X$ is a natural predual of $\mathcal F(M)$.
\end{proposition}

In what follows we provide a slightly different proof of Kalton's result, based now on Petun\={\i}n-Pl\={\i}\v cko theorem (see \cite{GodefroyPP,Petunin}). 
We recall that this last theorem asserts that a closed subspace $S \subset X^*$ of the dual of a separable Banach space $X$ is an isometric predual of $X$ (that is $S^* = X$) if, and only if, $S$ is composed of norm-attaining functionals and $S$ separates the points of $X$. 
The use of this theorem to produce preduals to free spaces has become quite common (see~\cite{DaletThesis,dal1, dal2,vector} and also our Examples~\ref{ex:AG} and~\ref{ex:NaturalNonNatural}).
The benefit of this proof is that it avoids the metrizability assumption of the topology $\tau$ present in Kalton's original exposition of this result. 

In the proof we will also need the following lemma which restates in a general framework the first step of Kalton's proof.
\begin{lemma}\label{lemma:lsc}
Let $(M,d)$ be a metric space such that there is a topology $\tau$ on $M$ and a subset $X \subset \Lip(M) \cap \mathcal C_\tau(M)$ which $1$-S.P.U. Then $d:(M,\tau)^2 \to \Real$ is l.s.c.
\end{lemma}

\begin{proof}
Let $\{x_\alpha\}$, $\{y_\alpha\}$ be $\tau$-convergent nets in $M$ with limits $x$ and $y$, respectively. 
Given $\varepsilon>0$, find $f\in X$ such that $f(y)-f(x)\geq d(x,y)-\varepsilon$ and $||f||_L=1$. Then
\[ d(x,y)-\varepsilon \leq \lim_{\alpha} f(y_\alpha)-f(x_\alpha)\leq \liminf_{\alpha} d(x_\alpha,y_\alpha)\]
and the arbitrariness of $\varepsilon$ yields the desired conclusion. 
\end{proof}

\begin{proof}[Proof of Proposition~\ref{prop:dualityKalton}]
First of all, according to Lemma \ref{lemma:lsc}, note that $d$ is $\tau$-l.s.c. 

Now, we need to verify the conditions of Petun\={\i}n and Pl\={\i}\v cko's theorem. 
First, since $M$ is bounded, we see that $S$ is a closed subspace of $\Lip(M)$. Second, $S$ is separating since it is a lattice and separates the points of $M$ uniformly (see \cite{Kalton04}[Proposition 3.4]). 

Finally it remains to show that $X$ is made of norm-attaining functionals. To this end, let $f\in S_X$ and take sequences $\{x_n\}$, $\{y_n\}$ in $M$ such that $\lim_n \frac{f(x_n)-f(y_n)}{d(x_n,y_n)}=1$. 
Note that $\inf_n d(x_n,y_n)=:\theta >0$ since $f\in \lip(M)$. 
By the compactness of $(M,\tau)$ and the boundedness of $d$, we can find subnets $\{x_\alpha\}$ of $\{x_n\}$ and $\{y_\alpha\}$ of $\{y_n\}$ such that $x_\alpha \stackrel{\tau}{\to} x$, $y_\alpha\stackrel{\tau}{\to} y$ and $d(x_\alpha,y_\alpha)\to C>0$. 
Then,
\[ 1 = \lim_\alpha\frac{f(x_\alpha)-f(x_\alpha)}{d(x_\alpha, y_\alpha) }
\to \frac{f(x)-f(y)}{C}\leq \frac{f(x)-f(y)}{d(x,y)}.\]
Thus $X$ is made up of norm-attaining functionals. 

To conclude, we get that $S$ is a natural predual by just applying Proposition \ref{prop:suffnatural}.
\end{proof}

The next proposition testifies that Kalton's theorem is the only way to build a natural predual if the predual is moreover required to be a subspace of little Lipschitz functions.

\begin{proposition}\label{p:NaturalLittleLip}
Let $M$ be a bounded metric space and let $X^*=\Free(M)$ be a natural predual such that $X \subseteq \lip(M)$. 
Then there exists a topology $\tau$ on $M$ such that 
$(M,\tau)$ is compact, 
the metric $d:(M,\tau)^2 \to \Real$ is l.s.c. 
and $X=\lip(M) \cap \mathcal C_\tau(M)$.
\end{proposition}
\begin{proof} 
We put $\tau:=\set{\delta^{-1}(U):U \in \sigma(\Free(M),X)}$.
Since $\delta(M)$ is $\sigma(\Free(M),X)$-closed 
and bounded, $(M,\tau)$ is compact.
Recall that $d(x,y)=\norm{\delta(x)-\delta(y)}$ and $\norm{\cdot}$ is $\sigma(\Free(M),X)$-lsc, so the metric $d$ is $\tau$-lsc.
Since 
\[
X=\set{x^*\in\Free(M)^*: x^* \mbox{ is } \sigma(\Free(M),X)-\mbox{continuous}}
\]
and $X \subset \lip(M)$, we get that $X\subseteq \lip(M,d)\cap \mathcal C_\tau(M)=:Y$. 
This means that $\sigma(\Free(M),Y)$ is stronger than $\sigma(\Free(M),X)$. 
On the other hand, Proposition~\ref{prop:dualityKalton} yields that $Y^*=\Free(M)$. Therefore, by compactness, $\sigma(\Free(M),X)$ and $\sigma(\Free(M),Y)$ coincide on $\closedball{\Free(M)}$.
As a consequence of Banach-Dieudonn\'e theorem, they coincide on $\Free(M)$. 
This means that 
\[
\begin{split}
X&=\set{x^* \in \Free(M)^*:x^* \mbox{ is }\sigma(\Free(M),X)-\mbox{continuous}}\\
&=\set{x^* \in \Free(M)^*:x^* \mbox{ is }\sigma(\Free(M),Y)-\mbox{continuous}}=Y.
\end{split}
\]
\end{proof}

But one should be aware that not all natural preduals are contained in the space of little Lipschitz functions.
\begin{example}
Let $M=\set{\frac1n:n\in \Natural} \cup \set{0}$ with the distance comming from the reals.
Then it is well known that $\Free(M)$ is isometrically isomorphic to $\ell_1$.
Further we know (Theorem 2.1 in~\cite{dal1}) that $\lip(M)$ is isometrically a predual.
Since $M$ is compact, every predual is natural.
So our Proposition~\ref{p:NaturalLittleLip} and the fact that $M$ is compact show that any isometric predual of $\ell_1$ which is not isometric to $\lip(M)$ intersects the complement of $\lip(M)$.
\end{example}

Note that $\Lip(M)=\lip(M)$ when $M$ is uniformly discrete. This observation and the previous results yield the following corollary.

\begin{corollary}\label{cor:unifdiscduality} 
Let $(M, d)$ be a uniformly discrete bounded separable metric space with $0 \in M$. 
Let $X$ be a Banach space.
Then it is equivalent:
\begin{itemize}
\item[(i)] $X$ is a natural predual of $\Free(M)$.
\item[(ii)] There is a Hausdorff topology $\tau$ on $M$ such that $(M, \tau)$ is compact, $d$ is $\tau$-l.s.c. and $X=\Lip(M,d) \cap \mathcal C_\tau(M)$ equipped with the norm $\norm{\cdot}_L$.
\end{itemize}
\end{corollary}

\begin{proof}
(ii) $\Rightarrow$ (i) 
Given $x,y\in M$, $x\neq y$, define $f\colon\{x,y\}\to \mathbb R$ by $f(x)=0$ and $f(y) = d(x,y)$. 
By Matouskova's extension theorem \cite{mat}, there is $\tilde{f}\in \Lip(M)\cap \mathcal C_\tau(M)$ extending $f$ such that $||\tilde{f}||_L=1$. 
Thus, the hypotheses of Proposition \ref{prop:dualityKalton} are satisfied. 

The implication (i) $\Rightarrow$ (ii) is contained in Proposition~\ref{p:NaturalLittleLip}.
\end{proof}

In what follows we are going to develop yet another sufficient condition for an isometric predual to be natural with the goal to show that certain preduals constructed by Weaver in~\cite{weaverduality} are natural.

\begin{proposition}\label{c:SuffNatural}
Let $M$ be a uniformly discrete, bounded, separable metric space and let $X\subset \Lip(M)$ be a Banach space such that $X^*=\Free(M)$ isometrically.
If for every $x \in M\setminus\set{0}$ the indicator function $\indicator{\set{x}}$ belongs to $X$,
then $X$ is a natural predual of $\Free(M)$.
Moreover $0$ is the unique accumulation point of $(\delta(M),w^*)$ and $X$ is isomorphic to $c_0$.
\end{proposition}
The proof will be based on the following general fact.

\begin{lemma}\label{l:BasisFact}
Let $X,Y$ be Banach spaces such that $X^*=Y$ isometrically, $Y$ admits a bounded Schauder basis $\{u_n\}$ and the biorthogonal functionals $\{u_n^*\}$ belong to $X$.
Then $u_n \to 0$ weakly*.
\end{lemma}
\begin{proof}
We will show that every subsequence of $\{u_n\}$ admits a further subsequence that converges weakly* to 0.
So let us consider such subsequence. By the weak* compactness and separability, it admits a weak* convergent subsequence, let us call it $\{u_n\}$ again.
So we have $u_n \to u \in X$ weakly*.
But this means that for every $m \in \Natural$ we have 
\[
0=\lim_{n\to\infty}\duality{u_m^*,u_n}=\duality{u_m^*,u}.
\]
Thus $u=0$.
\end{proof}

\begin{proof}[Proof of Proposition \ref{c:SuffNatural}]
Since $M$ is bounded and uniformly discrete, the sequence $\sequence{\delta(x)}_{x \in M \setminus\set{0}}$ is a Schauder basis which is equivalent to the unit vector basis of $\ell_1$.
The biorthogonal functionals are exactly the indicator functions $\indicator{\set{x}}$ for $x\neq 0$.
Applying Lemma~\ref{l:BasisFact} we get that $\delta(M)$ is weak* closed and that $0$ is the unique w$^*$-accumulation point of $\delta(M)$. 
Let $\tau$ be the restriction of the $w^*$-topology to $M$.
Now Corollary~\ref{cor:unifdiscduality} yields that $X=\Lip(M) \cap C_\tau(M)$.
But, since $M$ is bounded an uniformly discrete, we have that $\Lip(M)$ are just all bounded functions that vanish at $0$. It follows immediately that $X=c_0(M\setminus \set{0})$.
\end{proof}

\begin{remark}
In \cite{weaverduality}, Weaver proved a duality result for \emph{rigidly locally compact} metric spaces. 
We recall that a locally compact metric space is said to be rigidly locally compact (see the paragraph before Proposition 3.3  in \cite{weaverduality}) if for every $r>1$ and every $x \in M$, the closed ball  $B(x,\frac{d(0,x)}r)$ is compact. 
The duality result of Weaver in particular implies that for a separable uniformly discrete bounded metric space $M$ which is rigidly compact, the space
\[ X = \left\{f\in \Lip(M) : \frac{f(\cdot)}{d(\cdot,0)}\in C_0(M)\right\}\]
is an isometric predual of $\Free(M)$.
Here $C_0(M)$ denotes the set of continuous functions which are arbitrarily small out of compact sets. 
Since it is obvious that the indicator functions $\indicator{\set{x}}$ belong to $X$, Proposition~\ref{c:SuffNatural} implies that $X$ is a natural predual of $\Free(M)$ and that $X$ is isomorphic to $c_0$.
This shows that in the case of uniformly discrete bounded spaces, Corollary~\ref{cor:unifdiscduality} covers the cases in which  Weaver's result ensures the existence of a predual.
\end{remark}

Moreover, there is a metric space which satisfies the hypotheses of Corollary \ref{cor:unifdiscduality} and which is not rigidly locally compact. 

\begin{example}
Let us consider the metric space $M = \{0,1\}\times \N$ equipped with the following distance: $d((0,n),(1,m))=2$ for $n,m \in \N$, and if $n\neq m$ we have $d((0,n),(0,m))=1$ and $d((1,n),(1,m))=1$. Then $M$ satisfies the assumptions of Corollary \ref{cor:unifdiscduality}. Indeed, declare $(0,1)$ to be the accumulation point of the sequence $\{(0,n)\}$, $(1,1)$ to be the accumulation point of the sequence $\{(1,n)\}$, and then declare all the other points isolated. Now independently of the choice of the distinguished point $0_M$, $M$ is not rigidly locally compact. For instance, say that $0_M = (0,n)$. Then for every $r>1$, the ball $B((1,1),d(0_M,(1,1))/r) = B((1,1),2/r)$ contains all the elements of the form $(1,m)$ with $m \in \N$. Consequently the considered ball is not compact, which proves that $M$ is not rigidly locally compact. 
\end{example}

\section{Extremal structure for spaces with natural preduals}\label{s:ExtStrucNatPred}

We are going to focus now on the extreme points in the free spaces that admit a natural predual.
Assuming moreover that the predual is a subspace of little Lipschitz functions we get an affirmative answer to one of our main problems. Note that this is an extension of Corollary 3.3.6 in \cite{weaver}, where it is obtained the same result under the assumption that $M$ is compact.

\begin{proposition}\label{prop:extmoleculesdual} Let $M$ be a bounded metric space. Assume that there is a subspace $X$ of $\lip(M)$ which is a natural predual of $\mathcal F(M)$. Then
\[ \ext(B_{\Free(M)})\subset \{ m_{xy} : x,y\in M, x\neq y\}.\]
\end{proposition}
\begin{proof}
By the separation theorem we have that $B_{\F(M)} = \cconv^{w^*}(V)$. Thus, according to Milman theorem (see \cite[Theorem 3.41]{Fabian}), we have $\ext(B_{\F(M)}) \subset \overline{V}^{w^*} $. So let us consider $\gamma \in \ext(B_{\F(M)})$. Take a net $\{m_{x_\alpha,y_\alpha}\}$ in $V$ which $w^*$-converges to $\gamma$. By $w^*$-compactness of $\delta(M)$, we may assume (up to extracting subnets) that $\{\delta(x_\alpha)\}$ and $ \{\delta(y_\alpha)\}$ converge to some $\delta(x)$ and $\delta(y)$ respectively. 

Next, we claim that we may also assume that $\{d(x_\alpha,y_\alpha)\}$ converges to $C> 0$. Indeed, since $M$ is bounded, we may assume up to extract a further subnet that $\{d(x_\alpha,y_\alpha)\}$ converges to $C \geq 0$. By assumption, there is $f \in X$ such that $\<f,\gamma \> >  \|\gamma\|/2 =1/2$. Since $f \in \lip(M)$, there exists $\delta >0$ such that whenever $z_1,z_2 \in M$ satisfy $d(z_1,z_2)\leq \delta$ then we have $|f(z_1) - f(z_2)|\leq \frac{1}{2} d(z_1,z_2)$. Since
\[ \lim_\alpha \< f ,  m_{x_\alpha,y_\alpha} \>=\< f , \gamma \> > \frac{1}{2},\]
there is $\alpha_0$ such that $\<f,m_{x_\alpha,y_\alpha}\>>1/2$ for every $\alpha>\alpha_0$. Thus $d(x_\alpha,y_\alpha)> \delta$ for $\alpha>\alpha_0$, which implies that $C\geq \delta>0$. 
Summarizing, we have a net $\{m_{x_\alpha, y_\alpha}\}$ which $w^*$-converges to $\frac{\delta(x)-\delta(y)}{C}$. So, by uniqueness of the limit, $\gamma = \frac{\delta(x)-\delta(y)}{C}$. Since $\gamma \in \ext(B_{\mathcal F(M)})\subset S_{\mathcal F(M)}$, we get that $C=d(x,y)$ and so $\gamma = m_{xy}$.  
\end{proof}

We have learned that a weaker version (for compact $M$) of the following proposition appears in the preprint~\cite{AG} for compact metric spaces. Our approach, which is independent of~\cite{AG}, also yields a characterisation of exposed points of $B_{\mathcal F(M)}$. 

\begin{corollary}\label{cor:extnatpredual} Let $M$ be a bounded \emph{separable} metric space. Assume that there is a subspace $X$ of $\lip(M)$ which is a natural predual of $\mathcal F(M)$.  Then given $\mu\in B_{\mathcal F(M)}$ the following are equivalent:
\begin{itemize}
\item[(i)] $\mu\in \ext(B_{\mathcal F(M)})$.
\item[(ii)] $\mu\in \exp(B_{\mathcal F(M)})$.
\item[(iii)] There are $x,y \in M$, $x\neq y$, such that $[x,y]=\{x,y\}$ and $\mu = m_{xy}$.
\end{itemize}
\end{corollary}

\begin{proof}
(i)$\Rightarrow$(iii) follows from Proposition \ref{prop:extmoleculesdual}. Moreover, (ii)$\Rightarrow$(i) is clear, so it only remains to show (iii)$\Rightarrow$(ii). To this end, let $x, y\in M$, $x\neq y$, be so that $[x,y]=\{x,y\}$. Consider 
\[ A =\{\mu\in B_{\Free(M)} : \<f_{xy}, \mu \>=1 \}. \]
We will show that $A=\{m_{xy}\}$ and so $m_{xy}$ is exposed by $f_{xy}$ in $B_{\Free(M)}$. Let $\mu\in \ext(A)$. Since $A$ is an extremal subset of $B_{\Free(M)}$, $\mu$ is also an extreme point of $B_{\Free(M)}$ and so $\mu\in V\cap A$. Recall that if $\<f_{xy}, m_{u,v}\>=1$ then $u,v\in [x,y]$, therefore $V\cap A = \{m_{xy}\}$. Thus $\ext(A)\subset \{m_{xy}\}$. Finally note that $A$ is a closed convex subset of $B_{\Free(M)}$ and so $A=\cconv(\ext(A))=\{m_{xy}\}$ since the space $\Free(M)$ has (RNP) as being a separable dual.
\end{proof}

It is proved in Aliaga and Guirao's paper \cite{AG} that if $(M,d)$ is compact, then a molecule $m_{xy}$ is extreme in $B_{\F(M)}$ if and only if it is preserved extreme if and only if $[x,y]= \{x,y\}$. Thus, if $\lip(M)$ $1$-S.P.U. (and thus $\F(M)=\lip(M)^*$), Proposition \ref{prop:extmoleculesdual} and Aliaga and Guirao's result provides a complete description of the extreme points: 
they are the molecules $m_{xy}$ such that $[x,y]=\{x,y\}$. 
It is possible to obtain the same kind of complete descriptions in some different settings as it is proved in the following result (see also Section \ref{section:unifdisc}).

\begin{proposition}
Let $(M,d)$ be a metric space for which there is a Hausdorff topology $\tau$ such that $(M,\tau)$ is compact and $d:(M,\tau)^2 \to \Real$ is l.s.c.
Let $0<p<1$ and let $(M,d^p)$ be the $p$-snowflake of $M$. 
Then given $\mu\in B_{\mathcal F(M)}$ the following are equivalent:
\begin{itemize}
\item[(i)] $\mu\in \ext(B_{\mathcal F(M,d^p)})$.
\item[(ii)] $\mu\in \strexp(B_{\mathcal F(M,d^p)})$.
\item[(iii)] There are $x,y \in M$, $x\neq y$, such that $\mu = m_{xy}$. 
\end{itemize}
\end{proposition}

Observe that under the hypotheses above it is not necessarily true that $\Free(M)$ is a dual space, but $\Free(M,d^p)$ already is.

\begin{proof}
$(iii) \Longrightarrow (ii)$. Let us fix $x \neq y \in M$. Since $0<p<1$, it is readily seen that $[x,y]=\{x,y\}$. Moreover it is proved in \cite[Proposition 2.4.5]{weaver} that there is a peaking function at $(x,y)$. Thus $m_{xy}$ is a strongly exposed point (\cite[Theorem 4.4]{gpr}). The implication $(ii) \Longrightarrow (i)$ is obvious. To finish, the implication $(i) \Longrightarrow (iii)$ follows directly from Proposition \ref{prop:extmoleculesdual} and the fact that $[x,y]=\{x,y\}$ for every $x \neq y \in M$.
\end{proof}

Next we will show that the extremal structure of a free space has impact on its isometric preduals.
If a metric space $M$ is countable and satisfies the assumptions of Proposition \ref{prop:extmoleculesdual}, then $\ext(B_{\mathcal F(M)})$ is also countable. Therefore, any isometric predual of $\mathcal F(M)$ is isomorphic to a polyhedral space by a theorem of Fonf \cite{Fonf78}, and so it is saturated with subspaces isomorphic to $c_0$. This applies for instance in the following cases. 

\begin{corollary} Let $M$ be a countable compact metric space. Then any isometric predual of $\Free(M)$ (in particular $\lip(M)$) is isomorphic to a polyhedral space.
\end{corollary}

\begin{corollary} Let $(M, d)$ be a uniformly discrete bounded separable metric space
such that $\Free(M)$ admits a natural predual. 
Then any isometric predual of $\Free(M)$ is isomorphic to a polyhedral space.
\end{corollary}

\section{The uniformly discrete case}\label{section:unifdisc}
We have already witnessed that in the class of uniformly discrete and bounded metric spaces, many results about $\Free(M)$ become simpler.
Yet another example of this principle is the following main result of this section.
\begin{proposition}\label{unifdisccharext} 
Let $(M,d)$ be a bounded 
uniformly discrete metric space. Then a molecule $m_{xy}$ is an extreme point of $B_{\mathcal F(M)}$ if and only if $[x,y]=\{x,y\}$. 
\end{proposition}

Also we will need the following observation, perhaps of independent interest: Since a point $x\in B_X$ is extreme if and only if $x \in \ext(B_Y)$ for every 2-dimensional subspace $Y$ of $X$, the extreme points of $B_{\Free(M)}$ are separably determined.
Let us be more precise.

\begin{lemma}\label{lemma:extsep} Assume that $\mu_0\in B_{\mathcal F(M)}$ is not an extreme point of $B_{\mathcal F(M)}$. Then there is a separable subset $N\subset M$ such that $\mu_0\in\mathcal F(N)$ and $\mu_0\notin\ext(B_{\mathcal F(N)})$.
\end{lemma}

\begin{proof}
Write $\mu_0 = \frac{1}{2}(\mu_1+\mu_2)$, with $\mu_1,\mu_2\in B_{\mathcal F(M)}$. We can find sequences $\{\nu_n^i\}$ of finitely supported measures such that $\mu_i = \lim_{n\to\infty} \nu_n^i$ for $i=0,1,2$. Let $N = \{0\}\cup\operatorname{supp}\{\nu_n^i\}$. Note that the canonical inclusion $\mathcal F(N)\hookrightarrow \mathcal F(M)$ is an isometry and $\nu_n^i\in \mathcal F(N)$ for each $n, i$. Since $\mathcal F(N)$ is complete, it is a closed subspace of $\mathcal F(M)$. Thus $\mu_0,\mu_1,\mu_2\in \mathcal F(N)$ and so  $\mu_0\notin\ext(B_{\mathcal F(N)})$.
\end{proof}

\begin{proof}[Proof of Proposition~\ref{unifdisccharext}]
Let $m_{xy}$ be a molecule in $M$ such that $[x,y]=\{x,y\}$ and assume that $m_{xy}\notin\ext(B_{\mathcal F(M)})$. By Lemma \ref{lemma:extsep}, we may assume that $M$ is countable. Write $M=\{x_n:n \geq 0\}$. 
Let $\sequence{e_n:n \geq 1}$ be the unit vector basis of $\ell_1$.  
It is well known that the map $\delta(x_n)\mapsto e_n$ for $n\geq 1$  defines an isomorphism from $\Free(M)$ onto $\ell_1$. 
Thus $\sequence{\delta(x_n):n\geq 1}$ is a Schauder basis for $\Free(M)$. 

Now, let $x,y\in M$, $x\neq y$ be such that $[x,y]=\{x,y\}$. 
Assume that $m_{xy}=\frac{1}{2}(\mu+\nu)$ for $\mu,\nu\in B_{\mathcal F(M)}$ 
and write $\mu = \sum_{n=1}^\infty a_n \delta(x_n)$. 
Fix $n\in \mathbb N$ such that $x_{n}\notin\{x,y\}$. 
Then, there is $\varepsilon_n>0$ such that
\[
(1-\varepsilon_n)\left(d(x,x_n)+d(x_n,y)\right)\leq d(x,y).
\]
Let $g_n=f_{xy}+\varepsilon_n \indicator{\set{x_n}}$, which is an element of $\Lip(M)$ since $M$ is uniformly discrete. 
We will show that $||g_n||_L \leq 1$. 
To this end, take $u,v\in M$, $u\neq v$. 
Since $\norm{f}_L\leq 1$, it is clear that $\abs{\< g_n, m_{uv}\>} \leq 1$ if $u,v\neq x_n$. 
Thus we may assume $v=x_n$. 
Therefore (c) in Lemma \ref{lemma:IKWfunction} yields that $\<f_{xy}, m_{uv}\>\leq 1-\varepsilon_n$ and so $\<g_n, m_{uv}\>\leq 1$. 
Exchanging the roles of $u$ and $v$, we get that $||g_n||_L\leq 1$. 
Moreover, note that
\[1 = \<g_n, m_{xy}\> = \frac{1}{2}(\<g_n,\mu\> + \<g_n,\nu\>) \leq 1 \]
and so $\<g_n, \mu\> = 1$. 
Analogously we show that $\<f_{xy}, \mu\> =1$. 
Thus  $a_n=\< \indicator{\set{x_n}}, \mu\>=0$. 
Therefore $\mu = a \delta(x) + b \delta(y)$ for some $a, b\in \mathbb{R}$. 
Finally, let $f_1(t):=d(t,x)-d(0,x)$ and $f_2(t):=d(t,y)-d(0,x)$. 
Then $||f_i||_L=1$ and $\<f_i,m_{xy}\>=1$, so we also have $\<f_i, \mu\>=1$ for $i=1,2$. 
It follows from this that $a=-b=\frac{1}{d(x,y)}$, that is, $\mu=m_{xy}$. 
This implies that $m_{xy}$ is an extreme point of $B_{\mathcal F(M)}$.
\end{proof}

Next we show that preserved extreme points are automatically strongly exposed for uniformly discrete metric spaces. 
Notice that, contrary to other results in this section, no boundedness assumption is needed.

\begin{proposition}\label{presestrunidisc}
Let $M$ be a uniformly discrete metric space. 
Then every preserved extreme point of $B_{\mathcal F(M)}$ is also a strongly exposed point. 
\end{proposition}

\begin{proof}
Let $x,y\in M$ such that $m_{xy}$ is a preserved extreme point of $B_{\mathcal F(M)}$. Assume that $m_{xy}$ is not strongly exposed. 
By Theorem 4.4 in \cite{gpr}, the pair $(x,y)$ enjoys property (Z). 
That is, for each $n\in\mathbb N$ we can find $z_n\in M\setminus\{x,y\}$ such that
$$d(x,z_n)+d(y,z_n)\leq d(x,y)+\frac{1}{n}\min\{d(x,z_n),d(y,z_n)\}.$$
Thus,
\[(1-1/n)(d(x,z_n)+d(y,z_n))\leq d(x,y)\]
so it follows from condition (ii) in Theorem \ref{th:charpreserved} that $\min\{d(x,z_n), d(y,z_n)\}\to 0$. Since $M$ is uniformly discrete, this means that $\{z_n\}$ is eventually equal to either $x$ or $y$, a contradiction. 
\end{proof}

Aliaga and Guirao proved in \cite{AG} that, in the case of compact metric spaces, every molecule which is an extreme point of $B_{\mathcal F(M)}$ is also a preserved extreme point. 
However, that result is no longer true for general metric spaces, as the following example shows.

\begin{example}\label{example:AG} Consider the sequence in $c_0$ given by $x_1=2e_1$, and $x_n=e_1+(1+1/n)e_n$ for $n\geq 2$, where $\{e_n\}$ is the canonical basis. Let $M=\{0\} \cup \{x_n : n\in \mathbb N\}$. This metric space is considered in \cite[Example 4.2]{AG}, where it is proved that the molecule $m_{0 x_1}$ is not a preserved extreme point of $B_{\mathcal F(M)}$. Let us note that this fact also follows easily from Theorem \ref{th:charpreserved}. 
Moreover, by Proposition \ref{unifdisccharext} we have that $m_{0 x_1}\in \ext(B_{\mathcal F(M)})$. 
\end{example}

On the other hand, if we restrict our attention to uniformly discrete bounded metric spaces satisfying the hypotheses of the duality result, then all the families of distinguished points of $B_{\mathcal F(M)}$ that we have considered coincide.  

\begin{proposition}\label{prop:unifdiscexteq} Let $(M,d)$ be a uniformly discrete bounded metric space such that $\Free(M)$ admits a natural predual. Then for $\mu \in \closedball{\Free(M)}$ it is equivalent:
\begin{itemize}
\item[(i)] $\mu\in \ext(B_{\mathcal F(M)})$.
\item[(ii)] $\mu\in \strexp(B_{\mathcal F(M)})$.
\item[(iii)] There are $x,y \in M$, $x\neq y$, such that $\mu = m_{xy}$ and $[x,y]=\{x,y\}$. 
\end{itemize}
\end{proposition}

\begin{proof}
(i) $\Rightarrow$ (iii) follows from Proposition \ref{prop:extmoleculesdual}.
Moreover, (ii)$\Rightarrow$(i) trivially. 
Now, assume that $\mu=m_{xy}$ with $[x,y]=\{x,y\}$. We will show that the pair $(x,y)$ fails property (Z) and thus $\mu$ is a strongly exposed point. Assume, by contradiction, that there is a sequence $\{z_n\}$ in $M$ such that
\[ d(x,z_n)+d(y,z_n) \leq d(x,y) + \frac{1}{n}\min\{d(x,z_n),d(y,z_n)\}.\]
and so
\[ (1-1/n)(d(x,z_n)+d(y,z_n)) \leq d(x,y). \]
The compactness with respect to the $w^*$-topology ensures the existence of a $w^*$-cluster point $z$ of $\{z_n\}$ ($M$ and $\delta(M) \subset \F(M)$ being naturally identified). Now, by the lower semicontinuity of the distance, we have
\[ d(x,z)+d(y,z) \leq \liminf_{n\to\infty} (1-1/n)(d(x,z_n)+d(y,z_n)) \leq d(x,y).\]  
Therefore, $z\in [x,y]=\{x,y\}$. Suppose $z=x$. Denote $\theta  = \inf\{d(u,v):u\neq v\}>0$. The lower semicontinuity of $d$ yields
\begin{align*}
 \theta + d(x,y) &\leq \liminf_{n\to\infty} (1-1/n)(\theta+d(y,z_n)) \\
 &\leq \liminf_{n\to\infty} (1-1/n)(d(x,z_n)+d(y,z_n)) \leq d(x,y),
\end{align*}
which is impossible. The case $z=y$ yields a similar contradiction. Thus the pair $(x,y)$ has property $(Z)$. 
\end{proof}

We now give some examples in which the preduals of $\Free(M)$ have interesting properties. 
The first one is a uniformly discrete and bounded metric space $M$ such that $\F(M)$ is isometric to a dual Banach space but cannot admit a natural predual. This example comes from \cite{AG}[Example 4.2] and has already been introduced in Example \ref{example:AG}. 

\begin{example}\label{ex:AG}
Consider the sequence in $c_0$ given by $x_0=0, x_1=2e_1$, and $x_n=e_1+(1+1/n)e_n$ for $n\geq 2$, where $\{e_n\}$ is the canonical basis. 
Let $M=\{0\} \cup \{x_n : n\in \mathbb N\}$. 
Then\\
a) $\F(M)$ does not admit any natural predual.\\
b) the space $X=\set{f \in \Lip(M): \lim f(x_n)=f(x_1)/2}$ satisfies $X^*=\Free(M)$.
 
Our Corollary~\ref{cor:unifdiscduality} guarantees that in order to prove a) it is enough to show that there is no compact topology $\tau$ on $M$ such that $d$ is $\tau$-l.s.c.
Assume that $\tau$ is such a topology. 
Then the sequence $\set{x_n}$ admits a $\tau$-accumulation point $x \in M$.
Since $d$ is $\tau$-l.s.c. we get that $x \in B(0,1) \cap B(x_1,1)$.
But this is a contradiction as the latter set is clearly empty.

For the proof of b) we will employ the theorem of Petun\={\i}n and Pl\={\i}\v{c}hko. 
The space $X$ is a clearly separable closed subspace of $\Free(M)^*$.
Further, a simple case check shows that for any $x\neq y \in M$, $y \neq 0$, the function $f(x)=0, f(y)=d(x,y)$ can be extended as an element of $X$ without increasing the Lipschitz norm. 
Thus since $X$ is clearly a lattice, Proposition~3.4 of \cite{Kalton04} shows that $X$ is separating. 
Finally, 
if $f \in X$ and 
\[
\frac{f(x_{n_k})-f(x_{m_k})}{d(x_{n_k},x_{m_k})}\to \norm{f}_L
\]
then without loss of generality the sequence $\set{m_k}$ does not tend to infinity. 
Passing to a  subsequence, we may assume that it is constant, say $m_k=m$ for all $k \in \Natural$.
If $\set{n_k}$ does not tend to infinity, then $\frac{f(x_i)-f(x_m)}{d(x_i,x_m)}=\norm{f}_L$ for some $i \neq m$.
Otherwise, since $f \in X$, we have
\[
\frac{f(x_{n_k})-f(x_{m})}{d(x_{n_k},x_{m})}\to \frac{\frac{f(x_1)}{2}-f(m)}{d(x_1,x_m)}.
\]
So in this case the norm is attained at $\frac{1}{d(x_1,x_m)}\left(\delta(x_1)/2-\delta(x_m)\right) \in \closedball{\Free(M)}$.
It follows that every $f\in X$ attains its norm.
Thus by the theorem of Petun\={\i}n and Pl\={\i}\v{c}hko, $X^*=\Free(M)$.
\end{example}

Next we show that $\Free(M)$ can actually have both natural and non-natural preduals.

\begin{example}\label{ex:NaturalNonNatural}
Let $M=\set{0} \cup \set{1,2,3,\ldots}$ be a graph such that the edges are couples of the form $(0,n)$ with $n\geq 1$.
Let $d$ be the shortest path distance on $M$.
Then it is obvious and well known that $\Free(M)$ is isometric to $\ell_1$.
Moreover $\Free(M)$ admits both natural and non-natural preduals.
Indeed, an example of a natural predual is $X=\set{f\in \Lip(M): \lim f(n)=f(1)}$ (this is immediate using Corollary~\ref{cor:unifdiscduality}). 
An example of a non-natural predual is $Y=\set{f \in \Lip(M): \lim f(n)=-f(1)}$.
We leave to the reader the verification of the hypotheses of the theorem of Petunin and Plichko.
\end{example}

Our last example shows that there are uniformly discrete bounded metric spaces such that their free space does not admit any isometric predual at all.
Such observation is relevant to the open problem whether $\Free(M)$ has (MAP) for every uniformly discrete and bounded metric space $M$ (see also Problem~6.2 in~\cite{godsurvey}). 
Using a well known theorem of Grothendieck, in order to get an affirmative answer it would be enough to show that $\Free(M)$ is isometrically a dual space. 
Our example shows that such a proof cannot work in general. 
Nevertheless, for $M$ in this example, $\F(M)$ enjoys the (MAP).
\begin{example} \label{ex:nondual}
Let $M=\set{0} \cup \set{1,2,3,\ldots} \cup \set{a,b}$ with the following distances:
\begin{align*}
d(0,n)&=d(a,n)=d(b,n)=1+1/n,\\
d(a,b)&=d(0,a)=d(0,b)=2, \text{ and }\\
d(n,m)&=1
\end{align*}
for $n, m\in \set{1,2,3,\ldots}$. Then $\Free(M)$ is not isometrically a dual space.

Indeed, let us assume that $\Free(M)=X^*$. 
Then some subsequence of $\sequence{\delta(n)}$ is w$^*$-convergent to some $\mu \in \Free(M)$.
We have $\norm{\mu-\delta(a)}\leq 1$ and $\norm{\mu}\leq 1$.
But Proposition~\ref{unifdisccharext} implies that $\delta(a)/2$ is an extreme point of $\closedball{\Free(M)}$.
This means that $\closedball{\Free(M)}(0,1) \cap \closedball{\Free(M)}(\delta(a),1)=\set{\delta(a)/2}$.
Similarly for $\delta(b)/2$.
Hence $\delta(a)/2=\mu=\delta(b)/2$. Contradiction.

Let us now prove that $\F(M)$ has the (MAP). Let $M_n:= \{0,a,b,1,\ldots,n \}$ and define $f_n$: $M \to M_n$ by $f_n(x)=x$ if $x \in M_n$ and $f(x)=n$ otherwise. The function $f_n$ is obviously a retraction from $M$ to $M_n$. Moreover a simple computation leads to $\|f_n\|_{L} \leq 1+ 1/n$. Let us denote $\tilde{f_n}$: $\F(M) \to \F(M_n)$ the linearisation of $f_n$ which is in fact a projection of the same norm: $\|\tilde{f_n}\| \leq 1+1/n$. Then define $P_n:=(1+1/n)^{-1}\tilde{f_n}$. Obviously, $\|P_n\| \leq 1$, $P_n$ is of finite rank and $\|P_n\gamma - \gamma\| \to 0$ for every $\gamma \in \F(M)$. Thus $\F(M)$ has the (MAP).
\end{example}

\section{Compact metric spaces} \label{s:compact}

In this section we focus on the case in which $M$ is a compact metric space and $\mathcal F(M)$ is the dual of $\lip(M)$. Recall that in this case all extreme points of $B_{\mathcal F(M)}$ are molecules by Corollary 3.3.6 in \cite{weaver}. We will show that indeed $\mathcal F(M)$ satisfies a geometrical property, namely being \emph{weak* asymptotically uniformly convex}, which implies in particular that the norm and the weak* topologies agree in $S_{\mathcal F(M)}$ and so every extreme point of the closed ball is also a denting point. 

If $X$ is a separable Banach space then the \emph{modulus of weak*-asymptotic uniform convexity} of $X^*$ can be computed as follows (\cite{BM10}):

\begin{equation*} \overline{\delta}_{X^*}^{*}(t) =\inf_{x^*\in B_{X^*}} \inf_{\substack{x^*_n\stackrel{w^*}{\to}0\\ \norm{x^*_n}\geq t}} \liminf_{n\to\infty} \norm{x^*+x^*_n}-1.
\end{equation*}

Recall that $X^*$ is said to be \emph{weak*-asymptotically uniformly convex} (weak*-AUC for short) if $\overline{\delta}_{X^*}^{*}(t)>0$ for each $t>0$.

\begin{proposition}\label{p:compactweakAUC} Let $M$ be a compact metric space. Assume that $\lip(M)$ is $1$-norming. Then $\mathcal F(M)$ is weak*-AUC.
\end{proposition}

For the proof we need the following easy lemma.  

\begin{lemma}\label{lemma:weakstarfin} Let $\{x_n^*\}\subset X^*$ be a weak*-null sequence such that $\Vert x_n^*\Vert\geq 1$ and $F\subset X^*$ be a finite dimensional subspace. Then $\liminf_{n\to\infty} d(x_n^*,F)\geq \frac12$. 
\end{lemma}

\begin{proof}[Proof of Proposition \ref{p:compactweakAUC}] We will use the same arguments as in the proof of Proposition 8 in \cite{Petitjean17}. Fix $t>0$ and take $\gamma\in S_{\mathcal F(M)}$ and a weak*-null sequence $\{\gamma_n\}$ such that $||\gamma_n||\geq t$ for every $n\in\mathbb N$. We will prove that  
\begin{equation}\label{eq:weakAUCfree} \liminf_{n\to\infty} ||\gamma+\gamma_n||\geq 1+\frac{t}{2}.
\end{equation}
We may assume that $\gamma$ is finitely supported. Pick $f\in \lip(M)$ with $||f||_L=1$ and $\<f,\gamma\> >1-\varepsilon$. Take $\theta>0$ such that $\sup_{d(x,y)\leq \theta}|f(x)-f(y)|\leq\varepsilon d(x,y)$. Pick $\delta< \frac{\varepsilon \theta}{2(1+\varepsilon)}$. By compactness, there exists a finite subset $E\subset M$ containing the support of $\gamma$ and such that $\sup_{y\in M} d(y,E)< \delta$. We have $\liminf_{n\to\infty} d(\gamma_n/t, \mathcal F(E))\geq\frac{1}{2}$ by Lemma \ref{lemma:weakstarfin}. Now, by Hahn-Banach theorem, there exist a sequence $\{f_n\}\subset (1+\varepsilon)B_{\Lip(M)}$ such that $f_n|_E = 0$ and $\liminf_{n\to\infty} \<f_n,\gamma_n\>\geq \frac{t}{2}$. Consider $g_n=f+f_n$. By distinguishing the cases $d(x,y)<\theta$ and $d(x,y)>\theta$, one can show that $||g_n||_{L} \leq 1+\varepsilon$. Now we have
\begin{align*}
\liminf_{n\to\infty} ||\gamma+\gamma_n||&\geq \liminf_{n\to\infty}\frac{1}{1+\varepsilon}\<g_n,\gamma+\gamma_n\> \\
&= \frac{1}{1+\varepsilon}\liminf_{n\to\infty} (\<f,\gamma\>+\<f,\gamma_n\> + \<f_n,\gamma\>+\<f_n,\gamma_n\>)\\
&\geq \frac{1}{1+\varepsilon}(1-\varepsilon + \frac{t}{2}-\varepsilon) 
\end{align*}
since $\gamma_n\stackrel{w^*}{\to}0$ and $f\in \lip(M)$. Letting $\varepsilon\to 0$ proves (\ref{eq:weakAUCfree}). It follows that $\overline{\delta}_{\mathcal F(M)}^*(t)\geq \frac{1}{2}t$ and so $\mathcal F(M)$ is weak*-AUC. 
\end{proof}

It is well-known and easy to show that if $X^*$ is weak*-AUC then then every point of the unit sphere has weak*-neighbourhoods of arbitrarily small diameter. This fact and the Choquet's lemma yield that if $x^*\in \ext(B_{X^*})$ then  then there are weak*-slices of $B_{X^*}$ containing $x^*$ of arbitrarily small diameter. That is, every extreme point of $B_{X^*}$ is also a weak*-denting point. 

\begin{corollary} Let $M$ be a compact metric space. Assume that $\lip(M)$ $1$-S.P.U. Then every extreme point of $B_{\mathcal F(M)}$ is also a denting point.
\end{corollary}

At this point one could be inclined to believe that the denting points and the strongly exposed points of $\closedball{\Free(M)}$ coincide, at least when $M$ is compact. 
We are going to give an example of a compact metric space for which the inclusion $\strexp(\closedball{\Free(M)}) \subset \ext(\closedball{\Free(M)^{**}})\cap \Free(M)$ is strict.

\begin{example}\label{ex:TreeCompact} Let $(T,d)$ be the following set with its real-tree distance
\[
 [0,1]\times \set{0} \cup \bigcup_{n=2}^\infty \set{1-\frac1n}\times \left[0,\frac{1}{n^2}\right].
\]
We will consider $(\Omega,d)$ as the set 
\[
 \set{(0,0),(1,0)} \cup \set{\left(1-\frac1n,\frac1{n^2}\right):n\geq 2}
\]
together with the distance inherited from $(T,d)$.
Let us call for simplicity $0:=x_1:=(0,0)$, $x_\infty:=(1,0)$ and $x_n:=(1-\frac1n,\frac1{n^2})$ if $n\geq 2$.

Since the couple $(x_\infty,0)$ is not a peak couple, the characterisation of the points in $\strexp(B_{\mathcal F(M)})$ given in \cite{gpr} yields that $\delta(x_\infty)$ is not a strongly exposed point of $\closedball{\Free(\Omega)}$. 
 Aliaga and Guirao~\cite{AG} have proved that for a compact $M$, the condition $[x,y]=\set{x,y}$ implies that $\frac{\delta(x)-\delta(y)}{d(x,y)}$ is a preserved extreme point of $\closedball{\Free(M)}$.
In particular $\delta(x_\infty)$ is a preserved extreme point of $\closedball{\Free(\Omega)}$. 
\end{example}

\section{Application to norm attainment}\label{Section:normattainment}
Given a metric space $M$ and a Banach space $X$, we have the following isometric identification $\Lip(M,X)=L(\mathcal F(M),X)$. Considering $f\in \Lip(M,X)$, we say that $f$ \textit{strongly attains its norm} if there are two different points $x,y\in M$ such that $\Vert f(x)-f(y)\Vert= \Vert f\Vert d(x,y)$. In view of the results of \cite{gpr,godsurvey,kms}, we wonder when the classical notion of norm attainment coincides with the one defined just above. In light of Bishop-Phelps theorem, we are also interested in the denseness of the class of Lipschitz functions which strongly attain their norm in $\Lip(M,X)$.

We will mean by $\mathrm{Lip}_{SNA}(M,X)$ (respectively $NA(\F(M),X)$) to the class of all functions in $\Lip(M,X)$ which strongly attain its norm (respectively which attain its norm as a linear and continuous operator from $\mathcal F(M)$ to $X$). Let us recall that a Banach space is said to have the \emph{Krein-Milman property} (KMP) if every non-empty closed convex bounded subset has an extreme point. It is well known that (RNP) implies (KMP), although the converse is still an open question (there are important classes of spaces for which the answer is yes). 

We shall begin by stating the scalar case of previous result.

\begin{proposition} \label{normattainreal}
Let $(M,d)$ be a metric space such that $\F(M)$ has (KMP) and such that $\mathrm{ext}(B_{\F(M)}) \subseteq V$. Then every $f \in \Lip (M)$ which attains its norm on $\F(M)$ also strongly attains it. In other words, the following equality holds: 
$$NA(\F(M),\mathbb R)= \justLip_{SNA} (M,\mathbb R).$$
Therefore, $\overline{\justLip_{SNA}(M,\mathbb R)}^{\|\, \cdot \,\|} = \Lip(M)$.
\end{proposition}

\begin{proof}
Notice that the inclusion $\justLip_{SNA}(M,\mathbb R)\subseteq NA(\mathcal F(M),\mathbb R)$ always holds. Thus we just have to prove the reverse one. Let $f$ be a function in $\Lip(M)$ which attains its norm on $B_{\F(M)}$. Since $\F(M)$ has (KMP), $f$ also attains its norm on an extreme point. Indeed, the set
\[ C=\{\mu\in B_{\mathcal F(M)} : \<f, \mu\> =1 \}\]
is a non-empty closed convex bounded subset of $\mathcal F(M)$, so there is $\mu\in \ext(C)$. Then it is easy to check that $\mu$ is also an extreme point of $B_{\mathcal F(M)}$. Since $\mathrm{ext}(B_{\F(M)}) \subseteq V$, $f$ attains its norm on a molecule $\frac{\delta(x)-\delta(y)}{d(x,y)}$ with $x \neq y$.

The last part follows from Bishop-Phelps theorem.
\end{proof}

As a consequence of Proposition \ref{prop:extmoleculesdual}, we get the following. 
\begin{corollary}
Let $M$ be a separable bounded metric space such that $\Free(M)$ admits a natural predual $X \subset \lip(M)$. Then $$NA(\F(M),\mathbb R)= \justLip_{SNA} (M,\mathbb R) \mbox{ and } \overline{\justLip_{SNA}(M,\mathbb R)}^{\|\, \cdot \,\|} = \Lip(M).$$
\end{corollary}

We give some examples where the previous corollary applies. 

\begin{example}
~
\begin{enumerate}
\item $M$ compact metric space such that $\lip(M)$ separates points uniformly (note that this result was first proved by Godefroy using M-ideal theory, see \cite{godsurvey}). For instance $M$ being compact and countable (see \cite{DaletThesis}), being the middle third Cantor set (see \cite{weaver}), or being any compact metric space where the distance is composed with a nontrivial gauge (see \cite{Kalton04}).
\item $M$ uniformly discrete metric space satisfying the assumptions of Proposition  \ref{prop:unifdiscexteq}.
\item $(B_{X^*},\|\cdot\|^p)$ unit ball of a separable dual Banach space where the distance is the norm to the power $p \in (0,1)$ (see Proposition 6.3 in \cite{Kalton04}).
\end{enumerate}
\end{example}

Following the ideas of the last section in \cite{vector} we can extend this kind of result in the vector valued case. \begin{proposition}
Let $X$ be a Banach space and $M$ be a separable bounded metric space such that $\Free(M)$ admits a natural predual $S \subset \lip(M)$. Assume that $X^*$ and $\F(M)$ has (RNP), and that $X^*$ or $\F(M)$ has $(AP)$. Then 
$$\overline{NA (\F(M),X^{**})}^{\|\, \cdot \, \|} =  \mathcal{L}(\F(M),X^{**}).$$ 
Equivalently, $\overline{\justLip_{SNA} (M,X^{**})}^{\|\, \cdot \, \|}= \Lip(M,X^{**})$.
\end{proposition}

\begin{proof}
First of all, note that $f \in \Lip(M,X^{**})$ strongly attains its norm if and only if $f$: $\F(M) \to X^{**}$ attains its operator norm. Indeed, assume that there is $\gamma \in \F(M)$, $\|\gamma\| \leq 1$ and $\|f(\gamma)\|_{X^{**}} = \|f\|_{L}$. By the Hahn-Banach theorem there is $z \in X^{***}$ such that $\< z , f(\gamma) \> = \|f(\gamma)\|$. Since $z \circ f$: $M \to \R$ is a real valued Lipschitz function which attains its operator norm on $\gamma$, Proposition \ref{normattainreal} provides a molecule $m_{xy} \in \F(M)$ such that $f$ attains its operator norm on it. This exactly means that $f$ strongly attains its norm. The other implication is trivial.

Next, the following isometric identification are well known (see \cite{vector} for example): $\Lip(M,X^{**}) = \mathcal{L}(\F(M),X^{**}) = (\F(M) \pten X^{*})^*$. Moreover, according to \cite[Lemma 4.3]{vector}, if $f \in \mathcal{L}(\F(M),X^{**}) = (\F(M) \pten X^{*})^*$ attains its norm as a linear form on $\F(M) \pten X^{*}$, then it also attains its norm as an operator in $\mathcal{L}(\F(M),X^{**})$.

To finish, apply Bishop-Phelps theorem to get the desired norm-denseness. 
\end{proof}

\textbf{Acknowledgment. } The authors are grateful to Ram\'on Aliaga and Antonio Guirao for sending them their preprint. The first and the last authors are grateful to the Laboratoire de Mathématiques de Besançon for the excellent working conditions during their visit in June 2017. 
The third author is grateful to Departamento de An\'alisis Matem\'atico of Universidad de Granada for hospitality and excelent working conditions during his visit in June 2017.
The authors would like to thank Mat\'ias Raja for useful conversations.

\bibliographystyle{siam}
\bibliography{biblio}

\def\cprime{$'$}
\begin{thebibliography}{10}

\bibitem{albiackalton}
{\sc F.~Albiac and N.~J. Kalton}, {\em Lipschitz structure of quasi-{B}anach
  spaces}, Israel J. Math., 170 (2009), pp.~317--335.

\bibitem{AG}
{\sc R.~Aliaga and A.~Guirao}, {\em On the preserved extremal structure of
  {L}ipschitz-free spaces}.
\newblock arXiv:1705.09579, 2017.

\bibitem{BM10}
{\sc L.~Borel-Mathurin}, {\em Isomorphismes non lin\'eaires entre espaces de
  {B}anach}, PhD thesis, Universit\'e Paris 6, 2010.

\bibitem{Bourgin83}
{\sc R.~D. Bourgin}, {\em Geometric aspects of convex sets with the
  {R}adon-{N}ikod\'ym property}, vol.~993 of Lecture Notes in Mathematics,
  Springer-Verlag, Berlin, 1983.

\bibitem{DaletThesis}
{\sc A.~Dalet}, {\em \'Etude des espaces Lipschitz-libres}, PhD thesis,
  Universit\'e de Franche-Comt\'e, 2015.

\bibitem{dal1}
{\sc A.~Dalet}, {\em Free spaces over countable compact metric spaces}, Proc.
  Amer. Math. Soc., 143 (2015), pp.~3537--3546.

\bibitem{dal2}
\leavevmode\vrule height 2pt depth -1.6pt width 23pt, {\em Free spaces over
  some proper metric spaces}, Mediterr. J. Math., 12 (2015), pp.~973--986.

\bibitem{Fabian}
{\sc M.~Fabian, P.~Habala, P.~H{\'a}jek, V.~Montesinos, and V.~Zizler}, {\em
  Banach space theory}, CMS Books in Mathematics/Ouvrages de Math\'ematiques de
  la SMC, Springer, New York, 2011.
\newblock The basis for linear and nonlinear analysis.

\bibitem{Fonf78}
{\sc V.~P. Fonf}, {\em Massiveness of the set of extremal points of the unit
  sphere of some conjugate {B}anach spaces}, Ukrain. Mat. Zh., 30 (1978),
  pp.~846--849, 863.

\bibitem{vector}
{\sc L.~Garc\'ia-Lirola, C.~Petitjean, and A.~R. Zoca}, {\em On the structure
  of spaces of vector-valued {L}ipschitz functions}, Studia Math.,  (2017).
\newblock To appear.

\bibitem{gpr}
{\sc L.~Garc\'ia-Lirola, A.~Proch\'azca, and A.~R. Zoca}, {\em A
  characterisation of the {D}augavet property in spaces of {L}ipschitz
  functions}.
\newblock arXiv:1705.05145, 2017.

\bibitem{gr}
{\sc L.~Garc\'ia-Lirola and A.~Rueda~Zoca}, {\em Unconditional almost
  squareness and applications to spaces of {L}ipschitz functions}, J. Math.
  Anal. Appl., 451 (2017), pp.~117--131.

\bibitem{GodefroyPP}
{\sc G.~Godefroy}, {\em Boundaries of a convex set and interpolation sets},
  Math. Ann., 277 (1987), pp.~173--184.

\bibitem{godsurvey}
\leavevmode\vrule height 2pt depth -1.6pt width 23pt, {\em A survey on
  {L}ipschitz-free {B}anach spaces}, Comment. Math., 55 (2015), pp.~89--118.

\bibitem{GMZ14}
{\sc A.~J. Guirao, V.~Montesinos, and V.~Zizler}, {\em On preserved and
  unpreserved extreme points}, in Descriptive topology and functional analysis,
  vol.~80 of Springer Proc. Math. Stat., Springer, Cham, 2014, pp.~163--193.

\bibitem{ikw}
{\sc Y.~Ivakhno, V.~Kadets, and D.~Werner}, {\em The {D}augavet property for
  spaces of {L}ipschitz functions}, Math. Scand., 101 (2007), pp.~261--279.

\bibitem{ikw2}
\leavevmode\vrule height 2pt depth -1.6pt width 23pt, {\em Corrigendum to:
  {T}he {D}augavet property for spaces of {L}ipschitz functions}, Math. Scand.,
  104 (2009), p.~319.

\bibitem{kms}
{\sc V.~Kadets, M.~Mart\'{\i}n, and M.~Soloviova}, {\em Norm-attaining
  {L}ipschitz functionals}, Banach J. Math. Anal., 10 (2016), pp.~621--637.

\bibitem{Kalton04}
{\sc N.~J. Kalton}, {\em Spaces of {L}ipschitz and {H}\"older functions and
  their applications}, Collect. Math., 55 (2004), pp.~171--217.

\bibitem{mat}
{\sc E.~Matou{\v{s}}kov{\'a}}, {\em Extensions of continuous and {L}ipschitz
  functions}, Canad. Math. Bull., 43 (2000), pp.~208--217.

\bibitem{Petitjean17}
{\sc C.~Petitjean}, {\em Lipschitz-free spaces and {S}chur properties}, J.
  Math. Anal. Appl., 453 (2017), pp.~894--907.

\bibitem{Petunin}
{\sc J.~I. Petun\={\i}n and A.~N. Pl\={\i}\v{c}ko}, {\em Some properties of the
  set of functionals that attain a supremum on the unit sphere}, Ukrain. Mat.
  \v Z., 26 (1974), pp.~102--106, 143.

\bibitem{weaverduality}
{\sc N.~Weaver}, {\em Duality for locally compact {L}ipschitz spaces}, Rocky
  Mountain J. Math., 26 (1996), pp.~337--353.

\bibitem{weaver}
\leavevmode\vrule height 2pt depth -1.6pt width 23pt, {\em Lipschitz algebras},
  World Scientific Publishing Co., Inc., River Edge, NJ, 1999.

\end{thebibliography}

\end{document}